\documentclass[reqno,10pt,centertags, draft]{amsart}
\usepackage{hyperref}
\usepackage{mdframed}
\usepackage{amsmath,amsthm,amscd,amssymb,latexsym,upref,stmaryrd,
enumerate,color,verbatim,yfonts,mathrsfs}
\usepackage{tikz}
\usetikzlibrary{trees}
\newcommand*{\mailto}[1]{\href{mailto:#1}{\nolinkurl{#1}}}

\newcommand{\bbN}{{\mathbb{N}}}
\newcommand{\bbP}{{\mathbb{P}}}

\newcommand{\bbR}{{\mathbb{R}}}

\newcommand{\bbZ}{{\mathbb{Z}}}

\newcommand{\cA}{{\mathcal A}}
\newcommand{\cB}{{\mathcal B}}

\newcommand{\cD}{{\mathcal D}}
\newcommand{\cE}{{\mathcal E}}
\newcommand{\cF}{{\mathcal F}}
\newcommand{\cG}{{\mathcal G}}

\newcommand{\cM}{{\mathcal M}}

\newcommand{\cP}{{\mathcal P}}

\newcommand{\cV}{{\mathcal V}}

\newcommand{\bfi}{{\bf i}}

\DeclareMathOperator{\supp}{supp}

\DeclareMathOperator{\gen}{gen}

\DeclareMathOperator{\diag}{diag}

\newcommand{\no}{\notag}
\newcommand{\lb}{\label}

\newcommand{\wti}{\widetilde}

\newcommand{\bi}{\bibitem}

\let\geq\geqslant
\let\leq\leqslant

\makeatletter
\def\theequation{\@arabic\c@equation}
\allowdisplaybreaks
\numberwithin{equation}{section}

\newtheorem{theorem}{Theorem}[section]

\newtheorem{hypothesis}[theorem]{Hypothesis}

\theoremstyle{remark}

\begin{document}

\numberwithin{equation}{section}
\allowdisplaybreaks

\title[Anderson Localization for radial tree graphs]{Anderson Localization for Radial Tree Graphs With Random Branching Numbers}

\author[D.\ Damanik]{David Damanik}
\address{ Department of Mathematics, Rice University, Houston, TX 77005, USA}
\email{\mailto{damanik@rice.edu}}
%\email{gesztesyf@missouri.edu}
%\urladdr{\url{https://www.math.missouri.edu/people/gesztesy}}
%\urladdr{https://www.math.missouri.edu/people/gesztesy}

\author[S.\ Sukhtaiev]{Selim Sukhtaiev}
\address{ Department of Mathematics, Rice University, Houston, TX 77005, USA}
\email{\mailto{sukhtaiev@rice.edu}}
%\email{\mailto{sswfd@mail.missouri.edu}}
%\email{sswfd@mail.missouri.edu}

%\dedicatory{}

\date{\today}
%\subjclass[]{}

%%%%%%%%%%%%
%%%%%%%%%%%%
\begin{abstract}
We prove Anderson localization for the discrete Laplace operator on radial tree graphs with random branching numbers. Our method relies on the representation of the Laplace operator as the direct sum of half-line Jacobi matrices whose entries are non-degenerate, independent, identically distributed random variables with singular distributions.
\end{abstract}
%%%%%%%%%%%%
%%%%%%%%%%%%

\maketitle

%\vspace*{-3mm}
%{\scriptsize{\tableofcontents}}
%\normalsize

%%%%%%%%%%%%%%%%%%%%%%%%%%%%%%
%%%%%%%%%%%%%%%%%%%%%%%%%%%%%%
\section{Introduction}  \lb{intro}

The main goal of this paper is to prove Anderson localization for the discrete Laplace operator on rooted radial tree graphs with random branching numbers. These are tree graphs with a fixed vertex $o$, the root, such that every vertex $v$ at a distance $n$ from the root $o$ is connected with $b_n \geq 2$ vertices at a distance $n+1$ from $o$, cf. Figure \ref{pic}. Assuming that $\{b_n\}_{n=0}^{\infty}$ is a sequence of non-degenerate i.i.d. random variables we show that, almost surely, the Laplace operator $\Delta$, cf. \eqref{a1}, has pure point spectrum and admits a basis of exponentially decaying eigenfunctions. More concretely, our main result is the following Theorem.

\begin{theorem}\lb{AL}
	Suppose that $\{\Delta(\omega)\}_{\omega\in\Omega}$ is a family of Laplace operators on radial tree graphs with random branching numbers. Suppose that the branching numbers are given by non-degenerate independent identically distributed random variables
	\begin{equation}\no
	\{b_n(\omega)\}_{n=0}^{\infty}\subset \{2,..., d\}\text{\ for some fixed}\ d>2.
	\end{equation}
	Then $\Delta(\omega)$ exhibits Anderson localization at all energies. That is, almost surely, $\Delta(\omega)$ has pure point spectrum and possesses a basis of exponentially decaying eigenfunctions.
\end{theorem}

The spectral theory of Schr\"odinger operators on tree graphs has attracted a lot of attention; see, for example, \cite{ASW1, ASW2, ASW3, Br07, HiPo, K1, K2, K3, SST, SS, S} and references therein. When studying the effects of randomness, one can for example consider random geometry (and then one typically studies the Laplacian), or one can consider a random potential (and then one typically fixed regular geometry). The effects of disorder in the geometry of trees have been studied in \cite{ASW1, HiPo}. In \cite{ASW1} the authors consider trees with edge lengths given by $\ell_e(\omega)=\ell e^{\lambda \omega_e}$, where $\ell>0$ is fixed, and $\lambda\in[0,1]$ determines the strength of the disorder and $\{\omega_e\}_{e\in\cE}$ are i.i.d. random variables. It is proved in \cite{ASW1} that the absolutely continuous spectrum of the Laplace operator is continuous (in the sense of \cite[Theorem 1.1]{ASW1}) at $\lambda=0$ almost surely. In the same work it is conjectured that such a continuity property fails in the case of radial disorder. This conjecture was settled in \cite{HiPo}, where the authors showed that in the radial case the spectrum is almost surely pure point. In fact, they proved Anderson localization for the random length and random Kirchhoff models by essentially the same method. Our work is motivated by that of P.~Hislop and O.~Post. The random branching model considered in this paper naturally complements the two models considered in \cite{HiPo}. However, it is worth noting that the methods of \cite{HiPo} are not applicable in the present setting since they are based on spectral averaging and hence rely heavily on the assumption that the random variables are absolutely continuous. Of course, in the case of random branching numbers such a hypothesis cannot be made. We therefore turn to a recent work \cite{BuDaFi} (which was inspired by \cite{BoSc1}) that offers a new proof of Anderson localization for random Schr\"odinger operators on $\bbZ$. Their methods can be adapted to show localization for the random half-line Jacobi matrices \eqref{a2} that naturally arise in the context of Laplace operators on radial tree graphs. Namely, we show that almost surely for every generalized eigenvalue $E$, cf. \eqref{b27}, one has
\begin{equation}\no
\lim\limits_{n\rightarrow\infty}\frac{1}{n}\log \|\cM_n^E(\omega)\|=L(E),
\end{equation}
where $\cM_n^E(\omega)$ is the $n$-step transfer matrix, cf. \eqref{a12}, \eqref{a14}, and $L(E)$ is the Lyapunov exponent which is shown to be positive for $E\not=0$. Once we have established localization for these Jacobi matrices, we then show that the Laplace operator $\Delta$ almost surely has a basis of exponentially decaying eigenfunctions, with almost exponential decay rate $L(E)+\frac{\log2}{2}$.
\begin{figure}
	\tikzstyle{level 1}=[sibling angle=120]
	\tikzstyle{level 2}=[sibling angle=60]
	\tikzstyle{level 3}=[sibling angle=30]
	\begin{tikzpicture}[grow cyclic, level distance=10mm]
	%\put(20,87){{\tiny \text{eigenvalues}}}
	\draw[thin,dashed] (0,0) circle (1cm);
	\draw[thin,dashed] (0,0) circle (1.92cm);
	
	%\draw[thin,dashed] (0,0) circle (3cm);
	\node {$\bullet$}
	child {node {$\bullet$}child {node {$\bullet$}child {node {$\bullet$}}
			child {node {$\bullet$}}
			child {node {$\bullet$}}}
		child {node {$\bullet$}child {node {$\bullet$}}
			child {node {$\bullet$}}
			child {node {$\bullet$}}}}
	child {node {$\bullet$}child {node {$\bullet$}child {node {$\bullet$}}
			child {node {$\bullet$}}
			child {node {$\bullet$}}}
		child {node {$\bullet$}child {node {$\bullet$}}
			child {node {$\bullet$}}
			child {node {$\bullet$}}}}
	child {node {$\bullet$}child {node {$\bullet$}child {node {$\bullet$}}
			child {node {$\bullet$}}
			child {node {$\bullet$}}}
		child {node {$\bullet$}child {node {$\bullet$}}
			child {node {$\bullet$}}
			child {node {$\bullet$}}}}
	;
	\draw (.4,-.6) node {\tiny{gen 1}};
	\draw (1.1,-1.2) node {\tiny{gen 2}};
	\draw (2.4,-2) node {\tiny{gen 3}};
	\end{tikzpicture}
	\caption{\ The first three generations of a radial rooted tree graph $\Gamma$ with branching numbers $b_0=3,\ b_1=2, b_2=3$.}\lb{pic}
\end{figure}
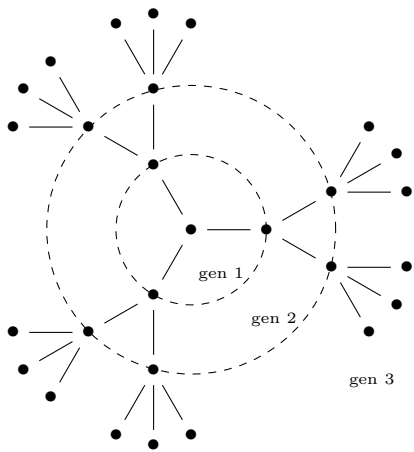

\section{The Laplacian on Radial Tree Graphs}\label{sec.2}
In this section we consider the Laplace operator on discrete rooted tree graphs $\Gamma=(\cV, \cE)$, with the set of vertices $\cV$ and the set of edges $\cE$. A tree graph $\Gamma$ is a graph without nontrivial closed paths. A tree graph $\Gamma$ together with a fixed vertex $o\in\cV$ is called rooted tree graph. The distance $d(u,v)$ between two vertices $u,v\in\cV$ is defined as the number of edges of the unique path connecting $u$ and $v$. The generation of a vertex $v$ is defined by $\gen(v):=d(o,v)$ and its branching number is given by $b(v):=\deg v-1$ if $v\not=o$, and $b(o)=\deg o$. A rooted tree graph is called radial if the branching number is a function of the generation only, that is $b(v)=b_n\geq 0$ whenever $\gen(v)=n$, for some sequence $\{b_n\}_{n=0}^{\infty}\subset \bbN$. In other words, cf. \cite[Definition 2.1]{Br07}, a rooted graph is radial if any vertex $v$ in generation $n$ is connected with $b_n$ vertices in generation $n+1$.

Assuming that the sequence $\{b_n\}_{n=1}^{\infty}$ is bounded, we define the Laplace operator by the formula
\begin{equation}\lb{a1}
(\Delta f)(v):=\sum_{u\in\cV\,:\,d(u,v)=1}f(u),\ f\in \ell^2(\Gamma).
\end{equation}
The boundedness of the sequence of branching numbers yields $\Delta\in\cB(\ell^2(\Gamma))$, moreover, since $\Delta$ is clearly symmetric, one has $\Delta=\Delta^*$ in $\ell^2(\Gamma)$.

An important step in the proof of Anderson localization is representing $\Delta$ as the direct sum of half-line Jacobi matrices given by
\begin{equation}\lb{a2}
J_n:=\begin{pmatrix}
0&\sqrt{b_{n}}&0&\ \\
\sqrt{b_{n}}&0&\sqrt{b_{n+1}}&\ddots \\
0&\sqrt{b_{n+1}}&\ddots&\ddots \\
\ &\ddots&\ddots&\ddots\\
\end{pmatrix},\ n\in\bbZ^+
\end{equation}
and acting in $\ell^2(\bbZ^+)$.

We assume that non-trivial branching occurs in every generation.
\begin{hypothesis}\lb{a3}
Suppose that for some $d\in\bbN$, $b_n\in\{2,..., d\}$, $n\in\bbZ^+$.
\end{hypothesis}

\begin{theorem}\lb{a10}{\ \cite[Theorem 2.4]{Br07}} Assume Hypothesis \ref{a3}. Let us define
\begin{equation}\lb{a5}
\beta_k:=
\begin{cases}
(b_{k-1}-1)\prod_{j=0}^{k-2}b_j,\ &k\geq 2,\\
b_0-1, &k=1,\\
1, &k=0.
\end{cases}
\end{equation}
Then  one has
\begin{equation}\lb{a6}
\ell^2(\Gamma)=\bigoplus_{N=0}^{\infty} \bigoplus_{k=1}^{\beta_N} H_{N,k},\ \Delta\upharpoonright_{H_{N,k}}\cong J_N.
\end{equation}
where every subspace $H_{N,k}$ reduces $\Delta$, for all admissible $N,k$.

% Moreover,
%\begin{equation}\lb{a6}
%\Delta\upharpoonright_{H_k}=\bigoplus_{i=1}^{\beta_k} J_i.
%\end{equation}
\end{theorem}
Concretely, there exists an orthonormal basis
\begin{equation}\no
\{\varphi_{N,k,j}\}_{N\in\bbZ^+,\, 1\leq k\leq \beta_N,\, j\in\bbZ^+}\text{\ in\ } \ell^2(\Gamma),
\end{equation}
so that

1) supp$(\varphi_{N,k,j})\subset \{ v\in\cV: \text{gen}(v)=N+j\}$, for all admissible $N,k,j$,

2) $\{\varphi_{N,k,j}\}_{j\in\bbZ^+}$ is an orthonormal basis for $H_{N,k}$, for all admissible $N,k$.

3) $\|\varphi_{N,k,j+1}\|_{L^{\infty}(\Gamma)}=b_j^{-1/2}\|\varphi_{N,k,j}\|_{L^{\infty}(\Gamma)}$ for all admissible $N,k,j$.

Associated with this basis is a unitary operator $\Phi$ given by
\begin{align}\lb{e1}
\Phi:  \ell^2(\Gamma)\rightarrow \bigoplus_{N=0}^{\infty} \bigoplus_{k=1}^{\beta_N} \ell^2(\bbZ^+);\ \varphi_{N,k,j}\mapsto \delta_{j}(N,k),\ j\in\bbZ^+,
\end{align}
where $\delta_j(N,k)=\delta_j$ should be viewed as a vector in the $N,k-$th copy of $\ell^2(\bbZ^+).$ Then one has
\begin{equation}\lb{f1}
\Phi \Delta \Phi^{-1}=\bigoplus_{N=0}^{\infty} \bigoplus_{k=1}^{\beta_N}  J_N(k),
\end{equation}
where $J_N(k)=J_N$ is the Jacobi matrix acting in the $N,k-$th copy of $\ell^2(\bbZ^+)$ and given by \eqref{a2}. %with $b_n$'s replaced by $\omega_n$'s.

%\subsection{Random Branching Model}

\section{Anderson Localization for Half-Line Jacobi Matrices}
In this section we turn to radial trees with branching numbers given by i.i.d. random variables which take values in the set $\cA=\{2,...,d\}$ for some fixed $d>2$. Suppose that $(\cA,\wti\mu)$ is a probability space and that $\supp \wti\mu$ contains at least two elements, $\#\supp \wti\mu\geq 2$. Furthermore, let us denote $\Omega:=(\supp \wti\mu)^{\bbZ^+}$, $\mu:=\wti\mu ^{\bbZ^+}$ and define the left shift
\begin{equation}\no
(T\omega)_n:=\omega_{n+1},\ \omega\in\Omega,\ n\in\bbZ^+.
\end{equation}
For a given $\omega\in\Omega$, $\Delta({\omega})$ denotes the Laplace operator on the radial tree graph $\Gamma_{\omega}$ with branching numbers defined by $b_n(\omega):=\omega_n$.

As discussed in Theorem \ref{a10}, the operator $\Delta({\omega})$ gives rise to a sequence of Jacobi matrices. The $n-$th matrix of this sequence is given by $J_0(T^n\omega)$, where the self-adjoint operator $J_0({\omega})$ is acting in $\ell^2(\bbZ^+)$ and given by
\begin{equation}\lb{a7}
[J_0({\omega})u](n):=
\begin{cases}
\sqrt{\omega_{n-1}}u(n-1)+\sqrt{\omega_{n}}u(n+1),\ &n\in\bbN, \\
\sqrt{\omega_{0}}u(1),&n=0.
\end{cases}
\end{equation}
A sequence $u=\{u_n\}_{n=0}^{\infty}$ satisfies
\begin{equation}\lb{a30}
\begin{cases}
\sqrt{\omega_{n-1}}u(n-1)+\sqrt{\omega_{n}}u(n+1)=Eu(n),\ &n\in\bbN, \\
\sqrt{\omega_{0}}u(1)=Eu(0),&
\end{cases}
\end{equation}
$E\in\bbR$ if and only if
\begin{equation}\lb{a12}
\begin{bmatrix}
u_{n+1}\\\sqrt{\omega_n}u_n
\end{bmatrix}=
\cM^E(T^{n-1}\omega)
\begin{bmatrix}
u_{n}\\\sqrt{\omega_{n-1}}u_{n-1}
\end{bmatrix},\
\cM^E(\omega):=
\begin{bmatrix}
\frac{E}{\sqrt{\omega_1}}& -\frac{1}{\sqrt{\omega_1}}\\
\sqrt{\omega_1}&0
\end{bmatrix},
\end{equation}
$n\geq 1.$
The map $\cM^{E}:\Omega\rightarrow \text{SL}(2,\bbR)$ determines an  $\text{SL}(2,\bbR)$-cocycle in a canonical way
\begin{equation}\no
(T,\cM^{E}): \Omega\times\bbR^2\rightarrow  \Omega\times\bbR^2,\ (T,\cM^{E})(\omega, v)=(T\omega, \cM^E(\omega)v).
\end{equation}
The iterates over the skew product are given by $(T,\cM^E)^n:=(T^n, \cM^E_n)$, where $\cM^E_0:=I_2$ and
\begin{equation}\lb{a14}
\cM^E_n(\omega)=\prod_{i=n-1}^0\cM^E(T^i\omega), n\in\bbN.
\end{equation}
The Lyapunov exponent of this cocycle is defined by
\begin{equation}\lb{a15}
L(E):=\lim\limits_{n\rightarrow\infty}\frac{1}{n}\int_{\Omega}\log \|\cM_n^E(\omega)\| \, d\mu(\omega).
\end{equation}
By Kingman's Subadditive Ergodic Theorem,
\begin{equation}\lb{a16.1}
L(E)=\lim\limits_{n\rightarrow\infty}\frac{1}{n}\log \|\cM_n^E(\omega)\|
\end{equation}
for $\mu-$almost every $\omega$.

Our proof of positivity of the Lyapunov exponent is based on the following facts.
\begin{theorem}\lb{a16}
Let $\nu$ be a probability measure on SL$(2,\bbR)$ that satisfies
\begin{equation}\no
\int \log\|M\| \, d\nu(M) < \infty.
\end{equation}
Let $G_{\nu}$ be the smallest closed subgroup of SL$(2,\bbR)$ that contains supp $\nu$.

(i) \cite[Theorem 8.6]{F63} Assume that $G_{\nu}$ is not compact and that it is strongly irreducible (cf. the definition preceding Theorem 2.1 in \cite{BuDaFi}). Then the Lyapunov exponent $L(\nu)$ associated with $\nu$ is positive.

(ii) \cite[Theorem B]{FuKi} Assume that the set
\begin{equation}\no
\emph{Fix}(G_{\nu}):=\{V\in\bbR\bbP^1: MV=V\text{\ for every\ } M\in G_{\nu}\}
\end{equation}
contains at most one element. If $\nu_k\rightarrow\nu$ weakly and boundedly $($cf. the definitions preceding Theorem 2.5 in \cite{BuDaFi}$)$, then $L(\nu_k)\rightarrow L(\nu)$ as $k\rightarrow\infty$.
\end{theorem}
A subgroup of SL$(2,\bbR)$ is called contracting if there exists a sequence $\{g_n\}_{n=1}^{\infty}$ of its elements such that $\|g_n\|g^{-1}_n$ converges to a rank-one matrix.
\begin{theorem}\lb{3.3} Let $\nu^{E}$ denote the push-forward measure of $\wti\mu$ under the map $\cM^{E}$, cf. \eqref{a12}. Suppose that $G_{\nu^E}$ is the smallest subgroup of SL$(2,\bbR)$ that contains $\supp \nu^{E}$.
Then $G_{\nu^E}$ is a non-compact, strongly irreducible, contracting subgroup for every $E\in\bbR\setminus\{0\}$. Moreover, one has
\begin{equation}\lb{a20}
\emph{Fix}(G_{\nu^E}):=\{V\in\bbR\bbP^1: MV=V\text{\ for every\ } M\in G_{\nu^E}\}=\emptyset,
\end{equation}
for all $E\in\bbR$. Furthermore, the Lyapunov exponent is a continuous function of the energy $E$. In addition, $L(E)>0$ if $E\not=0$ and $L(0)=0$.
\end{theorem}
\begin{proof}
For all $E\in\bbR$, $G_{\nu^E}$ contains at least $2$ distinct elements of the form
\begin{equation}\lb{a17}
M_{\alpha}:=\begin{bmatrix}
\frac{E}{\alpha}& -\frac{1}{\sqrt{\alpha}}\\
\sqrt{\alpha}&0
\end{bmatrix}.
\end{equation}
Let us pick arbitrary $\alpha\not=\beta$ and observe that the following sequence of matrices is unbounded and belongs to $G_{\nu^E}$
\begin{equation}\lb{a16.2}
A_n:=(M_{\alpha}M_{\beta}^{-1})^n=
\begin{bmatrix}
\alpha^{-\frac n2}\beta^{\frac n2}&0\\
0&\alpha^{\frac n2}\beta^{-\frac n2}\\
\end{bmatrix}, n\in\bbZ^+.
\end{equation}
Hence, $G_{\nu^E}$ is not compact.

Let us fix $E\not=0$. In order to prove that $G_{\nu^E}$ is strongly irreducible, we observe that for any $V\in\bbR\bbP^1$, one has
\begin{equation}\lb{a19}
\lim\limits_{n\rightarrow\pm\infty}A_nV\in\text{span}\{e_1\}\cup\text{span}\{e_2\},
\end{equation}
where $e_1=(1,0)^{\top}$, $e_2=(0,1)^{\top}$. Therefore, every finite subset $\cF\subset\bbR\bbP^1$ which is invariant under $G_{\nu^E}$ must be equal to one of the following sets:
\begin{equation}\lb{a21}
V_1:=\text{span}\{e_1\},\ V_2:=\text{span}\{e_2\},\ V_1\cup V_2.
\end{equation}
However, none of these sets is invariant when $E\not=0$. Indeed, in this case one has
\begin{equation}
M_{\alpha} V_1 \not\in \{ V_1,  V_2 \}, \quad M_{\alpha}^{-1} V_2 \not\in \{ V_1,  V_2 \}.
\end{equation}

To show \eqref{a20}, let us assume that there exists $V\in\text{Fix}(G_{\nu^E})$. Then $A_nV=V$ for each $n\in\bbN$, hence, using \eqref{a19} one infers that either $V=V_1$ or $V=V_2$. On the other hand, neither of $V_1, V_2$ belongs to $\text{Fix}(G_{\nu^E})$, since $M_\alpha V_k\not=V_k$ for $k\in{1,2}$.

By Theorem \ref{a16}, the Lyapunov exponent is continuous everywhere and positive at nonzero energies.

Next, we prove that  $L(0)=0$. To this end we explicitly compute the limit in \eqref{a16.1}. One has
\begin{align}
\begin{split}
\lb{a25}
\cM^0_{2n}(\omega)=\prod_{i=2n-1}^{0}\cM^0(T^i\omega)&=\prod_{i=n}^{1}\cM^0(T^{2i-1}\omega)\cM^0(T^{2i-2}\omega)\\
&=\begin{bmatrix}
\left(\prod_{i=n}^{1}\frac {\omega_{2i-1}}{\omega_{2i}}\right)^{1/2}&0\\
0&\left(\prod_{i=n}^{1}\frac {\omega_{2i}}{\omega_{2i-1}}\right)^{1/2}
\end{bmatrix}.
\end{split}
\end{align}
Denoting
\begin{equation}\lb{a27}
\xi_i:=\frac12\log(\omega_{2i-1}\omega_{2i}^{-1}),
\end{equation}
we notice that $\{\xi_i\}_{i=1}^{\infty}$ is a sequence of i.i.d. random variables satisfying $\int_{\cA}\xi_id\wti\mu=0$ for all $i\in\bbN$. By the law of large numbers, one has
\begin{equation}\lb{a26}
\frac{1}{n}\sum_{i=1}^{n}\xi_i\rightarrow0,\text{\ as\ }n\rightarrow \infty, \text{\ almost surely.}
\end{equation}
Combining \eqref{a25}, \eqref{a27}, and \eqref{a26}, one infers
\begin{equation}\no
L(0)=\lim\limits_{n\rightarrow\infty}\frac{1}{2n}\log \|\cM_{2n}^0(\omega)\|= 0,
\end{equation}
for $\mu-$almost every $\omega\in\Omega$.

Finally, $G_{\nu_E}$ is contracting since the $A_n$ converge to a rank-one matrix as $n\rightarrow\infty$.
\end{proof}

\begin{theorem}\lb{a40}
Assume Hypothesis \ref{a3}, then there exists a full-measure set $\Omega_0$ such that
\begin{equation}\lb{a41}
\sigma(J_0(\omega))=[-2\sqrt{d_{\mu}},2\sqrt{d_{\mu}}],\ \text{for all \ }\omega\in\Omega_0,
\end{equation}
where $d_{\mu}:=\max(\supp\wti\mu)$.
\end{theorem}

\begin{proof}
	Clearly, for every $\omega\in\Omega$ one has
	\begin{equation}\lb{a42}
	\sigma(J_0(\omega))\subset[-2\sqrt{d_{\mu}},2\sqrt{d_{\mu}}].
	\end{equation}
	Next, we prove that the opposite inclusion holds for $\mu-$a.e. $\omega\in \Omega$. To this end, let us consider the full-measure set  (cf., e.g., \cite[Proposition 3.8]{Ki08})
	\begin{equation}\lb{a43}
	\Omega_{0}:=\bigcap_{\delta>0}\bigcap_{R\in\bbZ^+}\bigcup_{k\in\bbZ^+}\{\omega\in\Omega: |\sqrt{\omega_j}-\sqrt{d_{\mu}}|<\delta\text{\ for all\ }k\leq j<k+R\}.
	\end{equation}
	Pick an arbitrary $E\in[-2\sqrt{d_{\mu}},2\sqrt{d_{\mu}}]$ and define $\theta = \theta(E) \in [0,\pi]$ by
	\begin{equation}\no
	E=\sqrt{d_{\mu}}(e^{\bfi \theta}+e^{-\bfi \theta}).
	\end{equation}
	For every $\omega\in\Omega_0$ we can find $k_l\in\bbZ^+$, $R_l\rightarrow\infty$, $\delta_l\rightarrow0$ as $l\rightarrow\infty$ such that
	\begin{equation}\no
	|\sqrt{\omega_j}-\sqrt{d_{\mu}}|<\delta_l\text{\ for all\ } k_l\leq j< k_l + R_l,\ l\in\bbN.
	\end{equation}
	The inclusion $E\in\sigma(J_0(\omega))$ follows from the Weyl Criterion provided
	\begin{equation}\lb{a45}
	\|(J_0(\omega)-E)\psi_l\|\rightarrow0,\ \ell\rightarrow\infty.
	\end{equation}
	where, for each $l\in\bbN$,
	\begin{equation}\lb{a44}
	\psi_{l}(j):=\begin{cases}
	R_l^{-1/2}e^{\bfi j\theta},\ &\text{for all\ }k_l\leq j< k_l + R_l,\\
	0,& \text{otherwise.}
	\end{cases}
	\end{equation}
	To prove \eqref{a45} let us notice that
	\begin{align}\no
	((J_0(\omega)-E)\psi_l)_j=\begin{cases}
	R_l^{-1/2}\omega_{j}^{\frac12}e^{\bfi (j+1)\theta},\hspace{1.78cm}\text{\ if\ } j=k_l-1,\\
	R_l^{-1/2}(\omega_{j}^{\frac12}e^{\bfi (j+1)\theta}-Ee^{\bfi j\theta}),\hspace{0.25cm}\text{\ if\ } j= k_l,\\
	%\sqrt{\omega_{k_l}}\psi(k_l+1),&j=k_l,\\
	R_l^{-1/2}e^{\bfi j\theta}((\omega_{j-1}^{\frac12}-d_{\mu}^{\frac{1}{2}})e^{-\bfi\theta}+(\omega_{j}^{\frac12}-d_{\mu}^{\frac12})e^{\bfi\theta}),\\
	\hspace{4.45cm}\text{\ if\ }k_l<j<k_l+R_l-1,\\
	R_l^{-1/2}(\omega_{j-1}^{\frac12}e^{\bfi (j-1)\theta}-Ee^{\bfi j\theta}),\text{\ if\ }j=k_l+R_l-1,\\
	%&j=k_l+R_l,\\
	R_l^{-1/2}\omega_{j-1}^{\frac12}e^{\bfi (j-1)\theta},\hspace{1.53cm}\text{\ if\ }j=k_l+R_l,\\
	0, \hspace{4.24cm}\text{otherwise}.
	\end{cases}
	\end{align}
	Hence,
	\begin{equation}\no
	\|(J_0(\omega)-E)\psi_l\|^2\lesssim_{d_{\mu}, E} R_l^{-1} + R_l^{-1}\sum_{j=k_l+1}^{k_l+R_l-2}(\omega_{j-1}^{\frac12}-d_{\mu}^{\frac{1}{2}})^2\underset{l\rightarrow\infty}{=}o(1).
	\end{equation}	
%To prove \eqref{a45} let us notice that
%	\begin{align}\no
%	((J_0(\omega)-E)\psi_l)_j=\begin{cases}
%	R_l^{-1/2}(\omega_{j}^{\frac12}e^{\bfi (j+1)\theta}-Ee^{\bfi j\theta}),\hspace{0.25cm}\text{\ if\ } j\in\{k_l-1, k_l\},\\
%	%\sqrt{\omega_{k_l}}\psi(k_l+1),&j=k_l,\\
%	R_l^{-1/2}e^{\bfi j\theta}((\omega_{j-1}^{\frac12}-d_{\mu}^{\frac{1}{2}})e^{-\bfi\theta}+(\omega_{j}^{\frac12}-d_{\mu}^{\frac12})e^{\bfi\theta}),\\
%	\hspace{4.45cm}\text{\ if\ }k_l<j<k_l+R_l-1,\\
%	R_l^{-1/2}(\omega_{j-1}^{\frac12}e^{\bfi (j-1)\theta}-Ee^{\bfi j\theta}),\text{\ if\ }j\in\{k_l+R_l-1, k_l+R_l\},\\
%	%&j=k_l+R_l,\\
%	0, \hspace{4.24cm}\text{otherwise}.
%	\end{cases}
%	\end{align}
%	Hence,
%	\begin{equation}\no
%	\|(J_0(\omega)-E)\psi_l\|^2\lesssim_{d_{\mu}, E} R_l^{-1} + R_l^{-1}\sum_{j=k_l+1}^{k_l+R_l-2}(\omega_{j-1}^{\frac12}-d_{\mu}^{\frac{1}{2}})^2\underset{l\rightarrow\infty}{=}o(1).
%	\end{equation}
\end{proof}
%It is convenient to write \eqref{a12} in terms of
%\begin{equation}\lb{a11}
%\begin{bmatrix}
%u_{n+1}\\u_n
%\end{bmatrix}=
%M^E(T^{n-1}\omega)
%\begin{bmatrix}
%u_{n}\\u_{n-1}
%\end{bmatrix},\
%M^E(\omega):=
%\begin{bmatrix}
%\frac{E}{\sqrt{\omega_1}}& -\frac{\sqrt{\omega_{0}}}{\sqrt{\omega_1}}\\
%\sqrt{\omega_1}&0
%\end{bmatrix},
%n\geq 1.
%\end{equation}

Next we focus on proving Anderson localization for $J_0(\omega)$. Our proof closely follows the main line of arguments from  \cite{BuDaFi}. Let $\cG(J_0(\omega))$ denote the set of energies for which \eqref{a30}  admits non-trivial solutions satisfying
\begin{equation}\lb{b27}
|u(n)|\leq C_u(1+n),\ C_u>0, n\in\bbZ^+.
\end{equation}
A key ingredient of the proof of Anderson localization is the following theorem.

\begin{theorem}\lb{a50}
There exists a full-measure set $\Omega_1\subset\Omega$, such that for all $\omega\in\Omega_1$ and every $E\in\cG(J_0(\omega))$, one has
\begin{equation}\lb{a51}
\lim\limits_{n\rightarrow\infty}\frac{1}{n}\log \|\cM_n^E(\omega)\|=L(E).
\end{equation}
\end{theorem}

Let \begin{equation}
I_l:=[-2\sqrt{d},2\sqrt{d}]\setminus(-1/l,1/l),\ l>1.
\end{equation}
Then Theorem \ref{a50} follows from a slightly weaker result, Theorem \ref{b16.1}, upon taking the intersection of all  $l$-dependent full measure sets from Theorem \ref{b16.1}. Hence, we proceed by discussing the latter theorem. The first item in the program is positivity and continuity of the Lyapunov exponent. By Theorem \ref{3.3}, the subgroup $G_{\nu^E}$ is noncompact, strongly irreducible, contracting, and $L(E)>0$ for all $E\in I_l$,\ $l>1$.

Thus, the results of \cite{BuDaFi} concerning the products of i.i.d. random SL$(2,\bbR)$ matrices are applicable in the present setting. To record these results let us introduce some notation. For a given $n\in\bbN$, let
\begin{equation}\lb{a52}
F_n(\omega,E):=\frac1n\log\|\cM^E_n(\omega)\|,
\end{equation}
%and let $P_n:\ell^2(\bbZ^+)\rightarrow\ell^2\{0,...,n-1\}$ denote the canonical restriction operator. Then %finite volume restriction of $J_0(\omega)$ is defined by the formula
\begin{equation}
J_0^n(\omega):=J_0(\omega){\upharpoonright_{[1,n]}}.
\end{equation}
For any $E\not\in\sigma(J_0^n(\omega))$, we set
\begin{equation}
G_{\omega,n}^E:=(J_0^n(\omega)-E)^{-1}\text{\ and\ }G_{\omega,n}^E(j,k):=\langle\delta_j, G_{\omega,n}^E\delta_k\rangle,\ j,k\in[1,n].
\end{equation}

\begin{theorem}\lb{b1}
Fix arbitrary $\varepsilon\in(0,1)$, $\zeta\in\bbZ^+$, and $l>1$. Then there exists a subset $\Omega_+(\varepsilon)\subset \Omega$ of full $\mu$-measure such that for all $\omega\in\Omega_+({\varepsilon})$, there exist $n_0=n_0(\omega,\varepsilon, l)$, $n_1=n_1(\omega,\varepsilon, l)$ so that the following statements hold.

$($i$)$ \cite[Proposition 5.2]{BuDaFi} For all   $n\geq \max(n_0, (\log(|\zeta|+1))^{2/3})$, and $E\in I_l$ one has
\begin{equation}\lb{a53}
\left|L(E)-\frac{1}{n^2}\sum_{s=0}^{n^2-1}F_n\left(T^{\zeta+sn}\omega, E\right)\right|<\varepsilon,
\end{equation}

$($ii$)$   \cite[Corollary 5.3]{BuDaFi} For all   $n\geq \max(n_1,\log^2(|\zeta|+1))$,  and $E\in I_l$ one has
\begin{equation}\lb{b14}
\frac{1}{n}\|\log\cM_n^E(T^{\zeta}\omega) \|\leq L(E)+2\varepsilon.
\end{equation}

$($iii$)$ For all $n\geq \varepsilon^{-1}\max(n_1,2\log^2(|\zeta|+1))$,  $E\in I_{l}\setminus \sigma(J_0(T^{\zeta}\omega))$, and $1\leq j,k\leq n$, one has
\begin{equation}\lb{b2}
\left|G^E_{T^{\zeta}\omega,n}(j,k)\right|\lesssim_{d} \frac{\exp[(n-|j-k|)L(E)+C_0\varepsilon n]}{|\det(J^n_0(T^{\zeta}\omega)-E)|}\prod_{i=0}^{n}\sqrt{\omega_{i+\zeta}},
\end{equation}
where $C_0=C_0(\wti\mu)>0$.
\end{theorem}

\begin{proof}
These results were obtained in \cite{BuDaFi} for Schr\"odinger operators on the whole line. Let us briefly point out a minor change in \eqref{b2}. The new term is the last factor in the right-hand side of \eqref{b2}. It  appears due to the following relations between the transfer matrices and the Green's functions of truncated Jacobi matrices, cf. \cite{JKS},
\begin{equation}\lb{b4}
\cM_n^E(\omega)=
\frac{1}{\prod_{i=0}^{n}\sqrt{\omega_i}}\cA(n)\begin{bmatrix}
\cP^E_n(\omega) & \sqrt{\omega_0}{\cP^E_{n-1}(T\omega)} \\
-\sqrt{\omega_n}\cP^E_{n-1}(\omega)&-\sqrt{\omega_n\omega_0}\cP^E_{n-1}(T\omega)
\end{bmatrix}\cA^{-1}(0),
\end{equation}
where
\begin{align}
&\cP^E_n(\omega):=\det([E-J_0(\omega)]_{\upharpoonright{[1,n]}}),\ n\geq1,\lb{c5}\\
& \cP^E_0(\omega):=0,\ \cP^E_{-1}(\omega):=1,\no\\
& \cA(n):=\diag\{1, \sqrt{\omega_n}\}, n\in\bbZ^+.\no
\end{align}
Moreover, one has
\begin{equation}\lb{b5}
G^E_{\omega,n}(j,k)=\frac{\cP_{j-1}^E(\omega)\cP_{n-k-1}^E(T^{k+1}\omega)}{\cP_{n}^E(\omega)}\prod_{k\leq i<j}\sqrt{\omega_{i}},
\end{equation}
where $1\leq j\leq k\leq n$ and the vacuous product that occurs for $j=k$ is defined to be equal to one. Using \eqref{b4}, \eqref{b5}, and \cite[(5.13)]{BuDaFi} and following the proof of \cite[Corollary 5.3]{BuDaFi}, one infers \eqref{b2}.
\end{proof}

A crucial element of the proof of Anderson localization is the elimination of double resonances. We recall it in the following theorem.

\begin{theorem}\lb{b19} \cite[Proposition 6.1]{BuDaFi}
Let us fix arbitrary $\varepsilon>0$, $N\in\bbN$, and define
\begin{equation}\lb{b12}
\cD_N(\varepsilon,l):= \left\{
\omega\in\Omega\Bigg|\begin{aligned}
&\|G^E_{\omega, [0,n]}\|\geq e^{K^2} \text{\ and\ }\  |F_n(T^{r}\omega,E)|\leq L(E)-\varepsilon,\\
&\text{\ for some\  }K\geq N,\ 0\leq n\leq K^9, \, E\in I_l,\ \\
&\hspace{1.6cm}K^{10}\leq r\leq \overline{K}:= \lfloor K^{\log K} \rfloor
\end{aligned}
\right\}
\end{equation}
Then there exist $C>0,\eta>0$ that do not depend on $N$ such that
\begin{equation}\lb{b15}
\mu(\cD_N(\varepsilon,l))\leq Ce^{-\eta N}.
\end{equation}
In particular,
\begin{equation}\lb{b16}
\Omega_-(\varepsilon):=\Omega\setminus\limsup\limits_{N\rightarrow\infty}\cD_N(\varepsilon,l),
\end{equation}
is a full-measure set.
\end{theorem}

\begin{theorem}\lb{b16.1}
For every $l>1$,  $E\in\cG(J_0(\omega))\cap I_l$ and $\mu-$almost every $\omega\in\Omega$, one has
\begin{equation}\lb{b18}
\lim\limits_{n\rightarrow\infty}\frac{1}{n}\log \|\cM_n^E(\omega)\|=L(E).
\end{equation}
\end{theorem}
\begin{proof}
Let us fix arbitrary  $l>1$,  $E\in\cG(J_0(\omega))\cap I_l$, and let us define the full-measure set
\begin{equation}\lb{b17}
\Omega_*:=\Omega_0\bigcap_{m\in\bbN}\Omega_+(m^{-1})\cap\Omega_-(m^{-1}),
\end{equation}
where $\Omega_0, \Omega_-(\cdot), \Omega_+(\cdot)$ are defined in Theorem \ref{a40}, Theorem \ref{b1}, Theorem \ref{b19} respectively.

Our next objective is to show that for all $\omega\in\Omega_*$, one has
\begin{align}
&\limsup\limits_{n\rightarrow\infty}\frac{1}{n}\log \|\cM_n^E(\omega)\|\leq L(E),\lb{b20}\\
&\liminf\limits_{n\rightarrow\infty}\frac{1}{n}\log \|\cM_n^E(\omega)\|\geq L(E).\lb{b21}
\end{align}
The first inequality follows from \eqref{b14}. The second one requires more subtle analysis. We follow the proof of \cite[Theorem 1.2]{BuDaFi} modifying some model-specific arguments.

Our goal is to show that for a given $\varepsilon>0$ and $\omega\in\Omega_*$ one has
\begin{equation}\lb{b22}
\liminf\limits_{n\rightarrow\infty}\frac{1}{n}\log \|\cM_n^E(\omega)\|\geq L(E)-\varepsilon.
\end{equation}
To this end, let $u$ be the generalized eigenfunction satisfying \eqref{a30}, \eqref{b27} and normalized by $u(0)=1$. Let
\begin{equation}
K:=\left\lceil\varepsilon^{-1}\max\{ n_0(\omega,\varepsilon), n_1(\omega,\varepsilon), N_0\}\right\rceil,\no
\end{equation}
where $n_0(\omega,\varepsilon)$ is defined in Theorem \ref{b1}, $n_1(\omega,\varepsilon)$ is the smallest natural number for which
\begin{equation}\lb{c30}
\omega\in \bigcap_{i\geq n_1(\omega,\varepsilon)}\big(\Omega_-(\varepsilon)\setminus D_i(\varepsilon, l)\big),
\end{equation}
and $N_0\geq 2$ is a sufficiently large number that will be defined later.

{\bf Step 1.} There exist $\{a,b\}\subset \bbZ^+$ such that $a\leq K^9$, $a+K^3-2\leq b\leq a+K^3$ such that
\begin{equation}\lb{c1}
\left|G^E_{\omega,[a,b]}(j,k)\right|\leq d^2 \exp(-|j-k|L(E)+C_0\varepsilon K^3),
\end{equation}
for all $j,k\in[a,b]$, and some $C_0>0$ dependent only on the measure $\wti \mu$.

Using \eqref{a53} with $n=K^3$, one has
\begin{equation}
\sum_{s=0}^{K^6-1}\frac{1}{K^6}\left(L(E)-\frac{1}{K^3}\log\|\cM^E_{K^3}(T^{sK^3}\omega)\|\right)<\varepsilon.
\end{equation}
Hence, for some $0\leq t\leq K^3(K^6-1)$, one has
\begin{equation}\lb{c2}
L(E)-\frac{1}{K^3}\log\|\cM^E_{K^3}(T^{t}\omega)\|<\varepsilon.
\end{equation}
Combining \eqref{b4}--\eqref{b5}, \eqref{c2}, and the fact that the norm of $\cM^E_{K^3}(T^{t}\omega)$ is at most four times its largest entry and assuming that $N_0\geq (\varepsilon^{-1}{\log(4d^2)})^{1/3}$, one infers
\begin{equation}\lb{c3}
L(E)-\frac{1}{K^3}\log\frac{\cP^E_{b-a}(T^{a}\omega)}{\prod_{i=0}^{K^3}\sqrt{\omega_{i+t}}}<2\varepsilon,
\end{equation}
for some $a\in\{t,t+1\}$ and $b\in\{K^3+a, K^3+a-1, K^3+a-2\}$. Next, we pick $N_0$ sufficiently large so that  $K^3\geq 2\log^2(1+a)$
to ensure applicability of \eqref{b2} with $\zeta$ replaced by $a$ and $n$ replaced by $b-a$. Using this modification of \eqref{b2} and \eqref{c3} we obtain
\begin{align}
\left|G^E_{T^{a}\omega,b-a}(j,k)\right|&\leq \frac{\exp[(b-a-|j-k|)L(E)+C_0\varepsilon (b-a)]}{\cP^E_{b-a}(T^{a}\omega)}\prod_{i=0}^{b-a}\sqrt{\omega_{i+a}}\\
&\leq d^2\frac{\exp[(K^3-|j-k|)L(E)+C_0\varepsilon K^3]}{\exp(K^3(L(E)-2\varepsilon))}\lb{c12}\\
&=d^2\exp[-|j-k|L(E)+(C_0+2)\varepsilon K^3],
\end{align}
for all $a\leq j,k\leq b$.

{\bf Step 2.} Let $\ell:=\lfloor \frac{a+b}{2}\rfloor$, then
\begin{equation}\lb{c21}
|u(\ell)|\leq d^{-\frac12} e^{-2K^2}.
\end{equation}
Indeed, utilizing \eqref{b27}, \eqref{c1}, and the standard representation of $u$ in terms of Green's function on $[a,b]$ and the boundary values $u(a-1), u(b+1)$, one infers
\begin{align}
|u(\ell)|&\leq \sqrt{\omega_{a-1}}|u(a-1)|\left|G^E_{\omega,[a,b]}(a,\ell)\right|+\sqrt{\omega_{b}}|u(b+1)|\left|G^E_{\omega,[a,b]}(\ell,b)\right|\no\\
&\leq d^{5/2} C_u(K^9+K^3+1)e^{C_0\varepsilon K^3}(e^{-|\ell-a|L(E)}+e^{-|b-\ell|L(E)})\no\\
&\leq 2d^{5/2} C_u(K^9+K^3+1)e^{-K^3L(E)/3+C_0\varepsilon K^3}< e^{-2K^2},\lb{c15}
\end{align}
where we assumed that $N_0$ is sufficiently large so that \eqref{c15} holds whenever $K\geq N_0$.

{\bf Step 3.} Let us recall the normalization $u(0)=1$. One therefore has
\begin{equation}\lb{c20}
1=u(0)\leq \sqrt{\omega_b}|u(\ell)| \left|G^E_{\omega,[0,\ell-1]}(0,\ell-1)\right|.
\end{equation}
Combining \eqref{c21}, \eqref{c20} and denoting $p:=\ell-1$ we get
\begin{equation}
\left\|G^E_{\omega,[0,p]}\right\|\geq e^{K^2}.
\end{equation}
Using the inclusion \eqref{c30}, $0\leq p\leq K^9$ and \eqref{b12}, we infer
\begin{equation}
\frac{1}{j}\log\|\cM^E_j(T^{r}\omega)\|>L(E)-\varepsilon.
\end{equation}
for all $j\in\{K,2K\}$, $r\in[K^{10}, \overline{K}]$. This fact together with the Avalanche principle yield
\begin{equation}
\frac{1}{m}\log \|\cM_m^E(\omega)\|\geq L(E)-\varepsilon,
\end{equation}
for all $m\in[K^{11}+K^{10}, \overline{K}]$ (cf. \cite[(6.17)--(6.19)]{BuDaFi}). Since every sufficiently large integer belongs to one of these intervals we obtain \eqref{b22}.
\end{proof}
A function $f\in L^2(\Gamma)$ is said to have almost exponential decay rate $L>0$ if for all $\varepsilon>0$ there exists $C_{\varepsilon}>0$ such that
\begin{equation}\no
|f(v)|\leq C_{\varepsilon}e^{-(L-\varepsilon)d(o,v)},\ v\in\Gamma.
\end{equation}

\begin{proof}[Proof of Theorem \ref{AL}]
First, we define the full measure set
\begin{equation}\lb{d1}
\wti\Omega:=\bigcap_{k=0}^{\infty}T^{-k}(\Omega_0\cap \Omega_1),
\end{equation}
where $\Omega_0,\Omega_1$ are defined in Theorem \ref{a40} and Theorem \ref{a50} respectively.
Then for all $\omega\in\wti\Omega$ and $k\in\bbZ^+$, one has
\begin{align}
&\sigma(J_0(T^k\omega))=\sigma(\Delta(\omega))=[-2\sqrt{d_{\mu}}, 2\sqrt{d_{\mu}}],\ d_{\mu}:=\max(\supp\wti\mu)\lb{d2}\\
&\sigma_c(J_0(T^k\omega))=\emptyset\lb{d3}.
\end{align}
It follows from \eqref{d3} (combined with the general direct sum decomposition given by Theorem~\ref{a10}) that
%Since the set of eigenvalues of $J_0(\omega)$ is dense in $[-2\sqrt{d_{\mu}}, 2\sqrt{d_{\mu}}]$ and every eigenvalue of $J_0(\omega)$ is an eigenvalue of $\Delta(\omega)$, one has
\begin{equation}\lb{d4}
\sigma_c(\Delta(\omega))=\emptyset,
\end{equation}
and in particular $\Delta(\omega)$ possesses a basis of eigenfunctions.

Next, we prove that this basis can be chosen so that each of its elements decays exponentially. To this end, let
\begin{equation}
\{u^{N,k,r}\}_{r=1}^{\infty}\subset \ell^2(\bbZ^+),
\end{equation}
be a basis of exponentially decaying eigenfunctions of $J_{N}(k)$ (where $J_{N}(k)$ is the $\omega$-dependent Jacobi matrix associated with $J_0(\omega)$ as in Section~\ref{sec.2}).
%Suppose that
%\begin{equation}
%\Delta(\omega) f=Ef,\ E\in[-2\sqrt{d_{\mu}},2\sqrt{d_{\mu}}],\ 0\not=f\in\ell^2(\Gamma).
%\end{equation}

It is sufficient to show that for arbitrary admissible $N,k,r$, the function $\Phi^{-1}u^{N,k,r}$ has almost exponential decay rate $L(E)+\frac{\log 2}{2}$ on the tree graph $\Gamma$.
%$\ker\{\Delta(\omega)-E\}$ is spanned by the functions with almost exponential decay rate $L(E)+\frac{\log %2}{2}$. Let $\{N_i\}_{i=0}^{\infty}\subset\bbN$ be the set of indices for which $E$ is an eigenvalue of %$J(T^{N_i}\omega)$ and let $u^{N_i}=\{u^{N_i}_j\}_{j=0}^{\infty}\in\ell^2(\bbZ^+)$ be the corresponding %eigenfunction. Then the function $u^{N_i,k}:=u^{N_i}$, i.e. $u^{N_i}$ considered in the $k-$th copy of %$\ell^2(\bbZ^+)$, $1\leq k\leq \beta_{N_i}$, gives rise to a function on $\Gamma$ that vanishes at the %vertices of the first $N_i-1$ generations, and has almost exponential decay rate $L(E)+\frac{\log 2}{2}$.
To that end, let us fix arbitrary vertex $v$, gen$(v)=m$ and notice that
\begin{align}\no
[\Phi^{-1} u^{N,k, r}] (v)=
\begin{cases}
\sum_{j=0}^{\infty}u^{N,k, r}_j\varphi_{N,k,j}(v)=u^{N,k, r}_{m-N}\varphi_{N,k,m-N}(v) & m\geq N,\\
0,&m<N,
\end{cases}
\end{align}
since $\varphi_{N,k,j}(v)=0$ whenever gen$(v)\not=N+j$. Furthermore, by the Osceledec Theorem and Theorem \ref{a50}, for all $\varepsilon>0$ there exists $C(\varepsilon)>0$ such that
\begin{align}
|u^{N,k, r}_j|\leq C(\varepsilon)e^{-(L(E)-\varepsilon)j},\ j\in\bbZ^+.
\end{align}
Hence, using
\begin{equation}\no
\|\varphi_{N,k,j+1}\|_{L^{\infty}(\Gamma)}=\frac{\|\varphi_{N,k,j}\|_{\ell^{\infty}(\Gamma)}}{\sqrt{\omega_j}} \leq \frac{\|\varphi_{N,k,j}\|_{\ell^{\infty}(\Gamma)}}{\sqrt2},
\end{equation}
we obtain
\begin{align}
\begin{split}
\big|[\Phi^{-1} u^{N,k, r}] (v)\big|&=|u^{N,k, r}_{m-N}\varphi_{N_i,k,m-N}(v)|\leq\frac{ C(\varepsilon)e^{(L(E)-\varepsilon)N}e^{-(L(E)-\varepsilon)m}}{2^{(m-N)/2}}\\
&\lesssim_N C(\varepsilon)e^{-\big(L(E)+\frac{\log2}{2}-\varepsilon\big)m}.
\end{split}
\end{align}
%Due to block-diagonal structure of $\Phi\Delta\Phi^{-1}$, one has
%\begin{equation}
%\ker\{\Delta(\omega)-E\}=\overline{\text{span}\{\Phi^{-1} u^{N_i,k}: i\in\bbN, 1\leq k\leq \beta_{N_i} \}}.
%\end{equation}
\end{proof}

%%%%%%%%%%%%%%%%%%%%%%%%%%%%%%%%
%%%%%%%%%%%%%%%%%%%%%%%%%%%%%%%%

\end{document}